\documentclass[12pt]{amsart}
\usepackage{amssymb,amsmath,amsthm,latexsym,mathtools,dsfont}
\usepackage{mathrsfs}
\usepackage{mdwlist}
\usepackage{graphicx}
\usepackage{xcolor}
\usepackage{enumerate}
\usepackage[shortlabels]{enumitem}
\usepackage{a4wide}
\input xy
\xyoption{all}
\usepackage{physics}%nicely handles subscripts at absolute value and modulus signs (when using the abs and norm commands)
\definecolor{amber}{rgb}{1.0, 0.75, 0.0}

\newcommand{\N}{\mathbb{N}}

\newcommand{\p}{\varphi}
\newcommand{\e}{\varepsilon}

\newcommand{\Ma}{\mathcal{M}}

\newcommand{\n}[1]{\|#1\|}
\newcommand{\nn}[1]{{\vert\kern-0.25ex\vert\kern-0.25ex\vert #1 
    \vert\kern-0.25ex\vert\kern-0.25ex\vert}}
\newcommand{\lnn}[1]{{\left\vert\kern-0.25ex\left\vert\kern-0.25ex\left\vert #1 
    \right\vert\kern-0.25ex\right\vert\kern-0.25ex\right\vert}}
\newcommand{\ccup}{\scalebox{0.85}{$\bigcup$}}

\newcommand{\dprime}{{\prime\prime}}
\renewcommand{\leq}{\leqslant}
\renewcommand{\geq}{\geqslant}

\newtheorem{theorem}{Theorem}

\newtheorem{proposition}[theorem]{Proposition}
\newtheorem{corollary}[theorem]{Corollary}

\theoremstyle{definition}
\newtheorem{definition}[theorem]{Definition}
\newtheorem{example}[theorem]{Example}

\theoremstyle{remark}
\newtheorem{remark}{Remark}

\definecolor{darkgreen}{RGB}{0, 153, 51}
\definecolor{violet}{RGB}{112, 73, 170}
\definecolor{darkred}{RGB}{153, 0, 0}
\definecolor{darkdarkblue}{RGB}{0, 0, 102}
\definecolor{darkblue}{RGB}{153, 204, 255}
\definecolor{bluee}{RGB}{204, 230, 255}
\definecolor{bluee2}{RGB}{128, 191, 255}
\definecolor{shadow}{RGB}{82, 122, 122}

\definecolor{my_green1}{RGB}{0, 153, 0}
\definecolor{my_orange}{RGB}{255, 204, 0}
\definecolor{my_yellow}{RGB}{255, 255, 102}

\title{Enveloping balls of Szlenk derivations}
\author{Tomasz Kochanek and Marek Miarka}
\address{Institute of Mathematics, University of Warsaw, Banacha~2, 02-097 Warsaw, Poland}
\email{tkoch@mimuw.edu.pl, m.miarka@uw.edu.pl}
\thanks{The work of the first-named author has been supported by the National Science Centre grant no. 2020/37/B/ST1/01052.}

\begin{document}
\begin{abstract}
For Banach spaces with a shrinking FDD, we provide estimates for the radii of the enveloping balls of the $\e$-Szlenk derivations of the dual unit ball. 
\end{abstract}
\maketitle

\section{Introduction}
\noindent
In this note, we are interested in estimation of the radii of the enveloping balls of $\e$-Szlenk derivations of the dual unit ball $B_{X^\ast}$ of a~Banach space $X$. Such derivations were the main ingredient of the definition of Szlenk index, introduced by Szlenk in \cite{szlenk}, where he famously showed that there is no universal Banach space in the class of separable reflexive spaces.

Recall that for a weak$^\ast$-compact set $K\subset X^\ast$, and any $\e\in (0,2)$, the $\e$-{\it Szlenk derivation} of $K$ is defined by
\begin{equation}\label{szlenk_def}
s_\e K=\bigl\{x^\ast\in K\colon \mathrm{diam} (K\cap V)>\e\mbox{ for every }w^\ast\mbox{-open neighborhood of }x^\ast\bigr\}.
\end{equation}
This is not exactly the original Szlenk's definition, but it is most commonly used nowadays. There is a~vast literature on studying the behavior of $\e$-Szlenk indices, that is, the rate of cutting out the dual unit ball by iterating the derivation $s_\e$; this study is connected with the notion of Szlenk power type (see, e.g., \cite{causey}, \cite{JFA}, \cite{PAMS}, \cite{GKL}, \cite{lancien}). However, in some situations, the knowledge about the exact shape of $s_\e B_{X^\ast}$ can be profitable. For example, as it was shown by C\'uth, Dole\v{z}al, Doucha and Kurka \cite{CDDK2}, the form of the $\e$-Szlenk derivation $s_\e B_{c_0^\ast}$ characterizes $c_0$ up to isometric isomorphism among separable $\mathscr{L}_{\infty,+}$-spaces.

\begin{remark}\label{limsup_R}
It is easy to see that for any weak$^\ast$-compact set $K\subset X^\ast$, and any $\e\in (0,2)$, we have $x^\ast\in s_\e K$ if and only there is a~net $(x_\alpha^\ast)_{\alpha\in A}\subset K$ which is weak$^\ast$-convergent to $x^\ast$ and such that for every $\alpha\in A$ there exist $\xi,\eta\in A$ with $\xi,\eta\geq\alpha$ and $\n{x_\xi^\ast-x_\eta^\ast}>\e$. Notice also that if in formula \eqref{szlenk_def} we allow all weak$^\ast$-open neighborhoods $V$ with $\mathrm{diam}(K\cap V)\geq\e$ (this convention is also quite often used), then the last inequality should be replaced by $\limsup_{\xi,\eta}\n{x_\xi^\ast-x_\eta^\ast}\geq\e$ (see \cite[Lemma~2.42]{HSVZ}).
\end{remark}

\begin{definition}
For any Banach space $X$ and $\e\in (0,2)$, we set
$$
r_X(\e)=\sup\bigl\{r>0\colon r B_{X^\ast}\subseteq s_\e B_{X^\ast}\bigr\}
$$
and
$$
R_X(\e)=\inf\bigl\{R>0\colon s_\e B_{X^\ast}\subseteq RB_{X^\ast}\bigr\}.
$$
\end{definition}

% Our main goal is to give lower/upper estimates for $r_X(\e)$ and $R_X(\e)$, respectively, for certain reflexive Banach spaces with unconditional bases. 
Our main goal is to provide lower/upper estimates for $r_X(\e)$ and $R_X(\e)$, respectively, for Banach spaces with a~shrinking FDD. In our examples, we will also consider sequential Orlicz spaces. Recall that $\mathcal{M}\colon [0,\infty)\to [0,\infty)$ is called a~{\it non-degenerate Orlicz function} if it has the following properties:
\begin{itemize}[leftmargin=24pt]
\setlength{\itemindent}{0mm}
\setlength{\itemsep}{2pt}
\item $\mathcal{M}(0)=0$ and $\mathcal{M}(t)>0$ for $t>0$,

\item $\mathcal{M}$ is continuous, convex and strictly increasing, 

\item $\lim\limits_{t\to\infty}\mathcal{M}(t)=\infty$.
\end{itemize}
Let $\ell_{\mathcal{M}}$ be the linear space consisting of all sequences $x =(x_n)^\infty_{n=1}\in \mathbb{R}^\mathbb{N}$ such that $\sum^\infty_{n=1}\mathcal{M}(\tfrac{\abs{x_n}}{\lambda})) <\infty$ for some $\lambda>0$. For $x\in \ell_{\mathcal{M}}$, define 
$$
\n{x}_{\mathcal{M}}=\inf\Big\{\lambda>0\colon \sum^\infty_{n=1}\mathcal{M}\Bigl(\frac{|x_n|}{\lambda}\Bigr)\leq 1\Big\}.
$$
Then, $(\ell_\Ma,\n{\cdot}_\Ma)$ is a Banach space; we call it a~sequential Orlicz space. It may happen that the canonical basic sequence $(e_n)_{n=1}^\infty$ does not span the whole of $\ell_\Ma$. It does if $\Ma$ satisfies the $\Delta_2$-condition at zero, and then $(e_n)_{n=1}^\infty$ is a~symmetric $1$-unconditional basis of $\ell_\Ma$ (for further details, see \cite[Ch.~4]{LT}).

%%%%%%%%%%%%%%%%%%%%%%%%%%%%%%%%%%%%%%%%%%%%%%%%%%%%%%%%%%%%%%%%%
%%%%%%%%%%%%%%%%%%%%%%%%%%%%%%%%%%%%%%%%%%%%%%%%%%%%%%%%%%%%%%%%%
%%%%%%%%%%%%%%%%%%%%%%%%%%%%%%%%%%%%%%%%%%%%%%%%%%%%%%%%%%%%%%%%%
\section{Results and examples}

\noindent
Recall that a~sequence $(E_n)_{n=1}^\infty$ of finite-dimensional subspaces of $X$ is called a~{\it finite-dimensional decomposition} (FDD, for short) of $X$ provided that for every $x\in X$ there is a~unique sequence $(x_n)_{n=1}^\infty$ such that $x_n\in E_n$ for each $n\in\N$ and $x=\sum_{n=1}^\infty x_n$. For any $n\in\N$, we define $Q_n\colon X\to \mathrm{span}\,\ccup_{j\leq n}E_j$ to be the $n^{\mathrm{th}}$ partial sum projection, that is, $Q_n(\sum_{i=1}^\infty x_i)=\sum_{i=1}^n x_i$ whenever $x_i\in E_i$ for every $i\in\N$. We denote by $c_{00}(\oplus_{n=1}^\infty E_n)$ the vector space of all sequences $(x_n)_{n=1}^\infty$ such that $x_n\in E_n$ for each $n\in\N$ and $x_n= 0$ for all but finitely many $n$'s. Let also $X^{(\ast)}$ stand for the norm closure of $c_{00}(\oplus_{n=1}^\infty E_n^\ast)$ in $X^\ast$, where the norm on each $E_n^\ast$ is the norm inherited from $X^\ast$ (this can be different than the norm induced as the dual of $E_n$). The FDD $(E_n)_{n=1}^\infty$ is called {\it shrinking} if $X^\ast=X^{(\ast)}$. In this case, $(E_n^\ast)_{n=1}^\infty$ is an~FDD of $X^\ast$.

%%%%%%%%%%%%%%%%%%%%%%%%%%%%%%%%%%%%%%%%%%%%%%%%%%%%%%%%%%%%%%%%%%
%%%%%%%%%%%%%%%%%%%%%%%%%%%%%%%%%%%%%%%%%%%%%%%%%%%%%%%%%%%%%%%%%%
\begin{theorem}\label{Thm_1}
Let $X$ be a Banach space with a~shrinking FDD $(E_n)_{n=1}^\infty$.
% \begin{enumerate}[label={\rm (\roman*)},leftmargin=24pt]
% \setlength{\itemindent}{0mm}
% \setlength{\itemsep}{1pt}
% \item $X$ has a shrinking FDD $(E_n)_{n=1}^\infty$;
% \item $X^\ast$ has an unconditional basis $(e_n^\ast)_{n=1}^\infty$ with the unconditional constant $K_u$.
% \end{enumerate}
Suppose that $Z$ is a~Banach space with $X^\ast\subseteq Z$ as sets, and that for some constants $0<C_1,C_2<\infty$, we have
\begin{equation}\label{C1_E}
C_1\n{x^\ast}_Z\leq\n{x^\ast}_{X^\ast}\leq C_2\n{x^\ast}_Z\quad\mbox{for each }\, x^\ast\in X^\ast.
\end{equation}
Assume that $\p,\psi,\chi\colon [0,\infty)\to [0,\infty)$ are continuous bijections such that for each $x^\ast\in X^\ast$ and all sufficiently large $n\in\N$, we have
\begin{equation}\label{psi_E1}
\p(\n{P_n(x^\ast)}_Z)+\psi(\n{(I-P_n)(x^\ast)}_Z)\leq\chi (\n{x^\ast}_Z),
\end{equation}
where $P_n\colon X^\ast\to X^\ast$ is the canonical projection onto $\mathrm{span}\,\ccup_{j\leq n}E_j^\ast$. Then, for every $\e\in (0,2)$, we have
\begin{equation}\label{R_E}
R_X(\e)\leq C_2\p^{-1}\Big\{\chi\Big(\frac{1}{C_1}\Big)-\psi\Big(\frac{\e}{2C_2}\Big)\Big\}.
\end{equation}
\end{theorem}

\begin{proof}
Fix an~arbitrary $\e^\prime\in (0,\e)$ and $x_0^\ast\in X^\ast$ with $\varrho(\e^\prime)<\n{x_0^\ast}_{X^\ast}\leq 1$, where
\begin{equation}\label{x_0_R}
\varrho(t)\coloneqq C_2\p^{-1}\Big\{\chi\Big(\frac{1}{C_1}\Big)-\psi\Big(\frac{t}{2C_2}\Big)\Big\}\quad (0<t<2).
\end{equation}
Choose $N\in\N$ so that $\n{P_N(x_0^\ast)}>\varrho(\e^\prime)$. For any, temporarily fixed, $\delta\in (0,1)$ we define a~weak$^\ast$-open neighborhood of $x_0^\ast$ as
\begin{equation*}\label{V}
V_\delta=\big\{ x^\ast\in X^\ast \colon \n{P_N(x^\ast-x_0^\ast)}_{X^\ast}<\delta\big\}.
\end{equation*}
Note that $V_\delta$ is indeed weak$^{\ast}$-open, as the projection $P_N$ is the adjoint operator, $P_N=Q_N^\ast$, where $Q_N\colon X\to X$ is the canonical projection onto $\mathrm{span}\,\ccup_{j\leq n}E_j$. Hence, $P_N$ is weak$^{\ast}$-to-norm continuous.

\vspace*{2mm}\noindent
\underline{{\it Claim.}} For every $x^\ast\in V_\delta\cap B_{X^\ast}$, we have
\begin{equation}\label{claim_E}
\n{(I-P_N)(x^\ast)}_{X^\ast}<C_2\!\cdot\! A(\delta),
\end{equation}
where
\begin{equation*}%\label{A_delta_def}
A(\delta)\coloneqq \psi^{-1}\Big\{\chi\Big(\frac{1}{C_1}\Big)-\p\Big(\frac{\varrho(\e^\prime)-\delta}{C_2}\Big)\Big\}.
\end{equation*}

\vspace*{4mm}
Indeed, observe that for any $x^\ast\in V\cap B_{X^\ast}$, we have 
$$
\n{P_N (x^\ast)}_{X^\ast}\geq \n{P_N (x_0^\ast)}_{X^\ast}-\n{P_N(x^\ast-x_0^\ast)}_{X^\ast}>\varrho(\e^\prime)-\delta.
$$
Therefore, applying inequality \eqref{psi_E1} (and enlarging $N$ if necessary), we get
\begin{equation*}
\begin{split}
\n{(I-P_N)(x^\ast)}_{X^\ast} & \leq C_2 \n{(I-P_N)(x^\ast)}_Z\\
&\leq C_2\psi^{-1}\Big\{\chi(\n{x^\ast}_Z)-\p(\n{P_N(x^\ast)}_Z)\Big\}\\
&\leq C_2\psi^{-1}\Big\{\chi\Big(\frac{\n{x^\ast}_{X^\ast}}{C_1}\Big)-\p\Big(\frac{\n{P_N(x^\ast)}_{X^\ast}}{C_2}\Big)\Big\}\\
&<C_2\psi^{-1}\Big\{\chi\Big(\frac{1}{C_1}\Big)-\p\Big(\frac{\varrho(\e^\prime)-\delta}{C_2}\Big)\Big\}=C_2\!\cdot\! A(\delta),
\end{split}
\end{equation*}
which completes the proof of our Claim.

\vspace*{1mm}
For any $y^\ast, z^\ast\in V_\delta\cap B_Y$ we have $\n{P_N(y^\ast-z^\ast)}_{X^\ast}<2\delta$. Hence, in view of our Claim, we get
\begin{equation}\label{y_z}
\begin{split}
\n{y^\ast-z^\ast}_{X^\ast} &\leq \n{P_N(y^\ast-z^\ast)}_{X^\ast}+\n{(I-P_N)(y^\ast)}_{X^\ast}\\
&\quad +\n{(I-P_N)(z^\ast)}_{X^\ast}<2\delta+2C_2\!\cdot\! A(\delta).
\end{split}
\end{equation}
Since for any $\e^\prime\in (0,2)$, we have $\varrho(\e^\prime)\in [0,C_2\p^{-1}(\chi(\tfrac{1}{C_1}))]$ and since the maps $\p$, $\psi^{-1}$ are uniformly continuous on any compact interval, we have
$$
A(\delta)=\psi^{-1}\Big\{\chi\Big(\frac{1}{C_1}\Big)-\p\Big(\frac{\varrho(\e^\prime)}{C_2}\Big)\Big\}+o(1)
$$
uniformly as $\delta\to 0^+$. Therefore, using \eqref{y_z} and the definition of $\varrho(\e^\prime)$ we get
$$
\n{y^\ast-z^\ast}_{X^\ast}<\e^\prime+O(\delta)\quad\mbox{for all }\, y^\ast, z^\ast\in V_\delta\cap B_{X^\ast}.
$$
In other words, there exist constants $D,\delta_1>0$ depending only on $\p,\psi,\chi,C_1,C_2$ such that for each $\delta\in (0,\delta_1)$, and any $y^\ast, z^\ast\in V_\delta\cap B_{X^\ast}$, we have $\n{y^\ast-z^\ast}_{X^\ast}<\e^\prime+D\delta$. Taking sufficiently small $\delta$, e.g. $\delta\leq\tfrac{1}{D}(\e-\e^\prime)$, we find a~weak$^\ast$-open neighborhood $V_\delta$ of $x_0^\ast$ with $\mathrm{diam}(V_\delta\cap B_{X^\ast})\leq\e$, which means that $x_0^\ast\not\in s_\e B_{X^\ast}$. Since $\e^\prime\in (0,\e)$ was arbitrary and $\lim_{\delta\nearrow\e}\varrho(\delta)=\varrho(\e)$, we conclude that no vector of norm larger than $\varrho(\e)$ can belong to $s_\e B_{X^\ast}$ which proves the assertion.
\end{proof}

%%%%%%%%%%%%%%%%%%%%%%%%%%%%%%%%%%%%%%%%%%%%%%%%%%%%%%%%%%%%%%%%%%
%%%%%%%%%%%%%%%%%%%%%%%%%%%%%%%%%%%%%%%%%%%%%%%%%%%%%%%%%%%%%%%%%%

%%%%%%%%%%%%%%%%%%%%%%%%%%%%%%%%%%%%%%%%%%%%%%%%%%%%%%%%%%%%%%%%%%
%%%%%%%%%%%%%%%%%%%%%%%%%%%%%%%%%%%%%%%%%%%%%%%%%%%%%%%%%%%%%%%%%%
\begin{theorem}\label{Thm_2}
Let $X$ be a Banach space satisfying one of the following conditions:

\begin{enumerate}[label={\rm (\roman*)},leftmargin=24pt]
\setlength{\itemindent}{0mm}
\setlength{\itemsep}{1pt}
\item $X$ has a shrinking FDD $(E_n)_{n=1}^\infty$;
\item $X^\ast$ has a semi-normalized weak$^{\,\ast}$-null basis $(e_n^\ast)_{n=1}^\infty$.
\end{enumerate}
Let also $Z$ be a~Banach space such that $X^\ast\subseteq Z$ as sets, and that for some $0<C_1,C_2<\infty$ inequality \eqref{C1_E} holds true. Assume that $\p,\psi,\chi\colon [0,\infty)\to [0,\infty)$ are continuous bijections such that the reversed version of inequality \eqref{psi_E1} holds true, i.e. for each $x^\ast\in X^\ast$ and all sufficiently large $n\in\N$, we have
\begin{equation}\label{psi_E2}
\chi(\n{x^\ast}_Z)\leq \p(\n{P_n(x^\ast)}_Z)+\psi(\n{(I-P_n)(x^\ast)}_Z),
\end{equation}
where $P_n\colon X^\ast\to X^\ast$ is the canonical projection onto  $\mathrm{span}\,\ccup_{j\leq n}E_j^\ast$ in the case where assumption {\rm (i)} holds true, or onto $[e_j^\ast]_{1\leq j\leq n}$ in the case where assumption {\rm (ii)} holds true. Then, for every $0<\e<2C_1\psi^{-1}(\chi(\tfrac{1}{C_2}))$, we have
%\vspace*{1mm}
\begin{equation}\label{r_E}
r_X(\e)\geq C_1\p^{-1}\Big\{\chi\Big(\frac{1}{C_2}\Big)-\psi\Big(\frac{\e}{2C_1}\Big)\Big\}.
\end{equation}
\end{theorem}

\begin{proof}
Take $0<\e<\e^\prime<\e^{\dprime}<2C_1\psi^{-1}(\chi(\tfrac{1}{C_2}))$ and fix any $x_0^\ast\in X^\ast$ with 
\begin{equation}\label{r_i_E}
\n{x_0^\ast}_{X^\ast}\leq \rho(\e^\dprime)\coloneqq C_1 \p^{-1}\Big\{\chi\Big(\frac{1}{C_2}\Big)-\psi\Big(\frac{\e^\dprime}{2C_1}\Big)\Big\}.
\end{equation}
In the case where assumption (i) holds true, we pick for every $n\in\N$, $e_n^\ast\in E_n^\ast$ with $\n{e_n^\ast}=1$. In the case where (ii) holds true, $(e_n^\ast)_{n=1}^\infty$ is as in the statement above. Define
\begin{equation}\label{mu_def}
\mu=\limsup_{n\to\infty}\n{e_n^\ast}_Z;
\end{equation}
notice that $0<\mu<\infty$ as $(e_n^\ast)_{n=1}^\infty$ is semi-normalized in case (ii) and, of course, $\mu=1$ in case (i). Notice also that in both cases we have $w^\ast$-$\lim_n e_n^\ast=0$.

Consider two sequences $(y_{n,+}^\ast)_{n=1}^\infty$ and $(y_{n,-}^\ast)_{n=1}^\infty$ in $X^\ast$ defined by
$$
y_{n,\pm}^\ast=P_n(x_0^\ast)\pm\frac{\e^\prime}{2C_1\mu}e_{n+1}^\ast\quad (n\in\N).
$$

\vspace*{2mm}\noindent
\underline{{\it Claim.}} For infinitely many $n\in\N$, we have
\begin{itemize}[leftmargin=24pt]
\setlength{\itemsep}{3pt}
\item[(i)] $\n{y_{n,+}^\ast-y_{n,-}^\ast}_{X^\ast}>\e$;

\item[(ii)] $\n{y_{n,\pm}^\ast}_{X^\ast}\leq 1$.
\end{itemize}

\vspace*{1mm}\noindent
Indeed, by definition \eqref{mu_def}, we have
\begin{equation*}
\begin{split}
\n{y_{n,+}^\ast-y_{n,-}^\ast}_{X^\ast} &\geq C_1 \n{y_{n,+}^\ast-y_{n,-}^\ast}_Z\\
&=\frac{\e^\prime}{\mu}\n{e_{n+1}^\ast}_Z>\frac{\e^\prime}{\mu}\!\cdot\!\frac{\e}{\e^\prime}\mu=\e
\end{split}
\end{equation*}
for infinitely many $n$'s. Similarly, for all sufficiently large $n\in\N$, we have $\n{e_{n+1}^\ast}_Z<\e^{\dprime}\mu/\e^\prime$. Using inequality \eqref{psi_E2} and the fact that $\lim_n P_n(x_0^\ast)=x_0^\ast$, for large enough $n$'s we get
\begin{equation*}
\begin{split}
\n{y_{n,\pm}^\ast}_Y &\leq C_2\, \n{y_{n,\pm}^\ast}_Z\\
&\leq C_2 \chi^{-1}\Big\{\p(\n{P_n(x_0^\ast)}_Z)+\psi\Big(\frac{\e^\prime}{2C_1\mu}\n{e_{n+1}^\ast}_Z\Big)\Big\}\\
&\leq C_2 \chi^{-1}\Big\{\p(\n{x_0^\ast}_Z)+\psi\Big(\frac{\e^\dprime}{2C_1}\Big)\Big\}\\
& \leq C_2 \chi^{-1}\Big\{\p\Big(\frac{1}{C_1}\n{x_0^\ast}_{X^\ast}\Big)+\psi\Big(\frac{\e^\dprime}{2C_1}\Big)\Big\}\leq 1,
\end{split}
\end{equation*}
where the last inequality follows from \eqref{r_i_E}. Consequently, our claim has been proved.

\vspace*{2mm}
Since $w^\ast$-$\lim_{n\to\infty} y_{n,\pm}^\ast=x_0^\ast$, we see that assertions (i) and (ii) show that for every weak$^\ast$-open neighborhood $U$ of $x_0^\ast$, we have $\mathrm{diam} (B_{X^\ast}\cap U)>\e$ which implies that $x_0^\ast\in s_\e B_{X^\ast}$. As $\e^\dprime$ was an arbitrary number larger than $\e$, we infer that 
$$
\bigcup_{\delta>\e}\rho(\delta) B_{X^\ast}\subseteq s_\e B_{X^\ast}.
$$
Since $\lim_{\delta\searrow\e}\rho(\delta)=\rho(\e)$ and the derivation $s_\e B_{X^\ast}$ is a~norm closed set, we conclude that $\rho(\e)B_{X^\ast}\subseteq s_\e B_{X^\ast}$ as desired.
\end{proof}
%%%%%%%%%%%%%%%%%%%%%%%%%%%%%%%%%%%%%%%%%%%%%%%%%%%%%%%%%%%%%%%%%%
%%%%%%%%%%%%%%%%%%%%%%%%%%%%%%%%%%%%%%%%%%%%%%%%%%%%%%%%%%%%%%%%%%

\begin{corollary}\label{C_psi}
Let $X$ be Banach space with a~shrinking FDD $(E_n)_{n=1}^\infty$. For $n\in\N$, let $P_n\colon X^\ast\to X^\ast$ be the canonical projection onto $\mathrm{span}\,\ccup_{j\leq n}E_j^\ast$ corresponding to the dual FDD $(E_n^\ast)_{n=1}^\infty$, and suppose that there exist continuous bijections $\p,\psi,\chi\colon [0,\infty)\to [0,\infty)$ of such that every $x^\ast\in X^\ast$ and all sufficiently large $n\in\N$, we have 
\begin{equation}\label{eq_E}
\p(\n{P_nx^\ast})+\psi(\n{(I-P_n)(x^\ast)})=\chi(\n{x^\ast})\quad (x^\ast\in X^\ast).
\end{equation}
Then for every $\e\in (0,2)$, we have 
$$
s_{\e} B_{X^\ast}=\Big\{x^\ast\in X^\ast\colon \n{x^\ast}\leq \p^{-1}\Big(\chi(1)-\psi\Big(\frac{\e}{2}\Big)\Big)\Big\}.
$$
\end{corollary}
\begin{proof}
We apply Theorems~\ref{Thm_1} and \ref{Thm_2} with $C_1=C_2=1$, and observe that the estimates for $R_X(\e)$ and $r_X(\e)$ given by \eqref{R_E} and \eqref{r_E} coincide.
\end{proof}

The next `folklore' result also follows directly from Theorems~\ref{Thm_1} and \ref{Thm_2}, and gives the form of the $\e$-Szlenk derivation of $\ell_p$-sums of finite-dimensional spaces.

\begin{corollary}\label{l_p_C}
Let $X=(\bigoplus_{n=1}^\infty E_n)_p$, where $(E_n)_{n=1}^\infty$ is any sequence of finite-dimensional spaces and $p\in (1,\infty)$. Then, for every $\e\in (0,2)$, we have 
$$
s_\e B_{X^\ast}=\Big(1-\Big(\frac{\e}{2}\Big)^q\Big)^{1/q}B_{X^\ast},\quad\,\mbox{where }\,\,\frac{1}{p}+\frac{1}{q}=1.
$$
\end{corollary}

\vspace*{1mm}
\begin{remark}%Orlicz
The above statement follows from Corollary~\ref{C_psi} and the fact that the functions $\p(t)=\psi(t)=\chi(t)=t^q$ satisfy equation \eqref{eq_E} for the FDD $(E_n^\ast)_{n=1}^\infty$ of any $\ell_q$-sum $X^\ast=(\bigoplus_{n=1}^\infty E_n^\ast)_q$. Interestingly, among the class of sequential Orlicz spaces, the power functions $\mathcal{M}(t)=t^q$ are the only ones for which equation \eqref{eq_E} is satisfied with $\p=\psi=\chi=\mathcal{M}$. This follows from the following proposition.
\end{remark}

\begin{proposition}
Suppose $\p\colon [0,\infty)\to [0,\infty)$ is a~non-degenerate Orlicz function such that
\begin{equation}\label{eq_P}
\p(\n{P_n(x)}_\p)+\p(\n{(I-P_n)(x)}_\p)=\p(\n{x}_\p)\quad\mbox{for all }\, x\in [e_n]_{n=1}^\infty\subseteq\ell_\p,\, n\in\N.
\end{equation}
Then, there is $q\in [1,\infty)$ such that $\p(t)=\p(1)t^q$ for $t\in [0,\infty)$.
\end{proposition}
\begin{proof}
We start by noticing that equation \eqref{eq_P} is equivalent to the following condition:

\begin{itemize}[leftmargin=20pt]
\item[($\star$)] For any finite sequences $(s_1,\ldots,s_k)$ and $(t_1,\ldots,t_l)$ of nonnegative numbers satisfying
$$
\sum_{i=1}^k\p(s_i)=1=\sum_{j=1}^l\p(t_j),
$$
and any $\lambda,\mu>0$, we have
\begin{equation}\label{claim_sum}
\sum_{i=1}^k\p\Bigg(\frac{\lambda s_i}{\p^{-1}(\p(\lambda)+\p(\mu))}\Bigg)+\sum_{j=1}^l\p\Bigg(\frac{\mu t_j}{\p^{-1}(\p(\lambda)+\p(\mu))}\Bigg)=1.
\end{equation}
\end{itemize}

\vspace*{2mm}\noindent
Indeed, consider any $x\in [e_n]_{n=1}^\infty$ of finite support, $x=\sum_{i=1}^k a_ie_i+\sum_{j=1}^l b_j e_{k+j}$, where not all of $a_i$'s and not all of $b_j$'s are zeros, and put $n=k$ in equation \eqref{eq_P}. Note that $\n{P_nx}_\p$ and $\n{(I-P_n)(x)}_\p$ are defined as the (uniquely determined) numbers $\lambda, \mu>0$ such that
$$
\sum_{i=1}^k\p\Big(\frac{\abs{a_i}}{\lambda}\Big)=1=\sum_{j=1}^l\p\Big(\frac{\abs{b_j}}{\mu}\Big).
$$
Then equation \eqref{eq_E} says that $\p^{-1}(\p(\lambda)+\p(\mu))=\n{x}_\p$ which is the unique positive number $\nu$ satisfying 
\begin{equation}\label{summ}
\sum_{i=1}^k\p\Big(\frac{\abs{a_i}}{\nu}\Big)+\sum_{j=1}^l\p\Big(\frac{\abs{b_j}}{\nu}\Big)=1.
\end{equation}
If we substitute $s_i=a_i/\lambda$ and $t_j=b_j/\mu$, then equality \eqref{summ} becomes \eqref{claim_sum}. This shows that equation \eqref{eq_E} implies ($\star$), since obviously $\lambda$ and $\mu$ could be taken as~arbitrary positive numbers. Conversely, since condition ($\star$) plainly implies \eqref{eq_P} for every $x$ of finite support, we obtain that equation for every element of $[e_n]_{n=1}^\infty$ by the continuity of $\p$.

\vspace*{1mm}
Set $\alpha=\p^{-1}(1)$ and define $F_\p\colon [0,\infty)^2\to [0,\infty)$ by $F_\p(s,t)=\p^{-1}(\p(s)+\p(t))$. Then, condition ($\star$) yields in particular (for $s_1=t_1=\alpha$) that
\begin{equation}\label{F_E}
F_\p\Big(\frac{\alpha s}{F_\p(s,t)},\frac{\alpha t}{F_\p(s,t)}\Big)=\alpha\quad\mbox{for all }\, s,t\in (0,\infty).
\end{equation}
For any $c>0$, define $S_c=\{(s,t)\in [0,\infty)^2\colon F_\p(s,t)=c\}$. Equation \eqref{F_E} then says that if $(s,t)\in S_c$, then $(\tfrac{\alpha s}{c},\tfrac{\alpha t}{c})\in S_\alpha$. This implies that for every $k>0$, $(ks,kt)\in S_{kc}$, i.e. the function $F_\p$ is homogeneous,
$$
F_\p(s,t)=k^{-1}F_\p(ks,kt)\quad\mbox{for all }\, k,s,t\in (0,\infty).
$$
In other words, for any fixed $k>0$, we have $F_\p=F_\psi$, where $\psi(x)=\p(kx)$ for $x\in [0,\infty)$. This can be rewritten in the form 
$$
s+t=\p\circ\psi^{-1}\big[\psi(\p^{-1}(s))+\psi(\p^{-1}(t))\big]\quad (s,t\in (0,\infty)),
$$
which simply means that the function $\Phi=\psi\circ\p^{-1}$ is additive,  i.e. $\Phi(s+t)=\Phi(s)+\Phi(t)$ for all $s,t\in (0,\infty)$. Since $\Phi$ is continuous, we have $\Phi(s)=\beta s$, that is $\psi(s)=\beta \p(s)$, for every $s\in (0,\infty)$ and some $\beta\in (0,\infty)$. Since $k>0$ was arbitrary, we conclude that there is a~function $\beta\colon (0,\infty)\to (0,\infty)$ such that $\p(ks)=\beta(k)\p(s)$ for all $k,s\in (0,\infty)$. From this it follows (see \cite[Thm.~13.3.4]{kuczma}) that $\p$ is proportional to a~(continuous) solutions $f$ of the equation $f(st)=f(s)f(t)$. By our assumption of convexity of $\p$, and by using \cite[Thm.~13.1.3]{kuczma}, we obtain $\p(t)\equiv \p(1)t^q$ with some $q\in [1,\infty)$.
\end{proof}

\begin{remark}
Corollary~\ref{C_psi} allows us to determine all iterates of the Szlenk derivation of $B_{X^\ast}$. Indeed, suppose $X$ satisfies the~assumptions of Corollary~\ref{C_psi} and let $r(\e)=\p^{-1}(\chi(1)-\psi(\tfrac{\e}{2}))$ for $\e\in (0,2)$. Using the formula
$$
s_\e (cK)=cs_{\e/c} (K)\quad\,\, (c>0),
$$
which is valid for any weak$^\ast$-compact set $K\subset X^\ast$ (and which follows easily from Remark~\ref{limsup_R}), we infer that for each $n$, $s_\e^n B_{X^\ast}$ is the ball centered at the origin and with radius $r_n(\e)$, where $(r_n(\e))_{n=1}^\infty$ is defined recursively as follows:
$$
r_{1}(\e)=r(\e),\quad r_{n+1}(\e)=\left\{\begin{array}{rl} r_n(\e)\!\cdot\!r\Big(\dfrac{\e}{r_n(\e)}\Big) & \mbox{if }\,\,r_n(\e)>\dfrac{\e}{2}\\[12pt]
0 & \mbox{if }\,\,r_n(\e)\leq\dfrac{\e}{2}.\end{array}\right.
$$
Therefore,
\begin{equation}\label{Sz_E}
Sz(X,\e)=1+\min\big\{n\in\N\colon r_n(\e)\leq\tfrac{\e}{2}\big\}.
\end{equation}

\noindent
For $X=(\bigoplus_{n=1}^\infty E_n)_p$, where each $E_n$ is finite-dimensional, using Corollary~\ref{l_p_C} we obtain by simple induction $r_n(\e)=(1-n(\tfrac{\e}{2})^q)^{1/q}$ for all $n\in\N$, $\e\in (0,2)$. Thus, \eqref{Sz_E} yields
$$
Sz\Big(\Big(\bigoplus_{n=1}^\infty E_n\Big)_{\!p},\,\e\Big)=\Big\lceil\Big(\frac{\e}{2}\Big)^{\!\!-q}\Big\rceil.
$$
\end{remark}

\begin{example}
For any $A\geq 0$ and $B>0$ consider the non-degenerate Orlicz function $\mathcal{M}_{A,B}(t)=At^4+Bt^2$. For any nonzero $x=(x_n)_{n=1}^\infty\in\ell_{\Ma_{A,B}}$, the norm $\n{x}_{\Ma_{A.B}}$ is \vspace*{-2pt}the unique solution $\lambda>0$ of $\sum_{n=1}^\infty \Ma_{A,B}(\tfrac{\abs{x_n}}{\lambda})=1$. Hence, $\lambda^4-B\n{x}_2^2\lambda^2-A\n{x}_4^4=0$ and an~elementary calculation shows that for every $x\in\ell_{\Ma_{A,B}}$ we have
$$
\sqrt{B}\n{x}_2\leq \n{x}_{\Ma_{A,B}}=\sqrt{\frac{B\n{x}_2^2+\sqrt{B^2\n{x}_2^4+4A\n{x}_4^4}}{2}}\leq \sqrt{\frac{B+\sqrt{B^2+4A}}{2}}\n{x}_2
$$
Denote the two constants occurring at the left- and right-hand side by $C_1$ and $C_2$, respectively. Obviously, $\Ma_{A,B}$ satisfies the $\Delta_2$-condition at zero, hence by \cite[Prop.~4.a.4]{LT}, $(e_n)_{n=1}^\infty$ is a~(symmetric) $1$-unconditional basis of $\ell_{\Ma_{A,B}}$. Since this space is isomorphic (but not isometrically isomorphic unless $A=0$) to $\ell_2$, it is reflexive, so we can apply Theorems~\ref{Thm_1} and \ref{Thm_2} to $X=\ell_{\Ma_{A,B}}^\ast$ and $Z=\ell_2$. This yields:
\begin{itemize}[leftmargin=24pt]
\setlength{\itemsep}{3pt}
\item $R_X(\e)\leq \sqrt{\frac{C_2^2}{C_1^2}-\frac{\e^2}{4}}=\sqrt{\frac{B+\sqrt{B^2+4A}}{2B}-\frac{\e^2}{4}}$\,\,\, for every $0<\e<2$;

\item $r_X(\e)\geq \sqrt{\frac{C_1^2}{C_2^2}-\frac{\e^2}{4}}=\sqrt{\frac{2B}{B+\sqrt{B^2+4A}}-\frac{\e^2}{4}}$\,\,\, for every $0<\e<\frac{2C_1}{C_2}$.
\end{itemize}

\vspace*{1mm}\noindent
Since $\lim_{A\to 0+} C_2=C_1$, we see that a~`stability' effect holds true: the Szlenk derivations of the unit balls of $X^\ast=\ell_{\Ma_{A,B}}$ approach the ball $s_\e B_{\ell_2}$ as $A$ tends to zero, that is,
$$
\lim_{A\to 0+} r_{\ell_{\Ma_{A,B}^\ast}}(\e)=\lim_{A\to 0+}R_{\ell_{\Ma_{A,B}^\ast}}(\e)=\sqrt{1-\frac{\e^2}{4}}.
$$
\end{example}

\vspace*{1mm}\noindent
{\bf Acknowledgement. }The first-named author acknowledges with gratitude the support from the National Science Centre, grant OPUS 19, project no.~2020/37/B/ST1/01052.


\begin{thebibliography}{99}

%\bibitem{AK} F. Albiac, N.J. Kalton, \emph{Topics in Banach space theory}, Graduate Texts in Mathematics~233, Springer, New York~2006.

\bibitem{causey} R.M. Causey, \emph{Power type $\xi$-asymptotically uniformly smooth norms}, Trans. Amer. Math. Soc.~{\bf 371} (2019), 1509--1546.

%\bibitem{CDDK} M. C\'uth, M. Dole\v{z}al, M.~Doucha, O.~Kurka, \emph{Polish spaces of Banach spaces}, Forum Math. Sigma~{\bf 10} (2022), Paper No. e26, 28 pp.

\bibitem{CDDK2} M. C\'uth, M. Dole\v{z}al, M.~Doucha, O.~Kurka, \emph{Polish spaces of Banach spaces: Complexity of isometry and isomorphism classes}, {\tt arXiv:2204.06834v1}

\bibitem{JFA} S. Draga, T. Kochanek,
\emph{Direct sums and summability of the Szlenk index}, J.~Funct. Anal.~{\bf 271} (2016), 642--671.

\bibitem{PAMS} S. Draga, T. Kochanek,
\emph{The Szlenk power type and tensor products of Banach spaces}, Proc. Amer. Math. Soc.~{\bf 145} (2017), 1685--1698. 

%\bibitem{DG} S. Dutta, A. Godard, \emph{Banach spaces with Property $(M)$ and their Szlenk indices}, Mediterr. J. Math.~{\bf 5} (2008), 211--220.

%\bibitem{FHHMZ}    M. Fabian, P. Habala, P. H\'{a}jek, V.~Montesinos, V.~Zizler, \emph{Banach Space Theory. The Basis for Linear and Nonlinear Analysis}, Springer, New York 2011. 

\bibitem{GKL} G. Godefroy, N.J. Kalton, G. Lancien, \emph{Szlenk indices and uniform homeomorphisms}, Trans. Amer. Math. Soc.~{\bf 353} (2001), 3895--3918.

\bibitem{HSVZ} P. H\'ajek, V. Montesinos Santaluc\'{i}a, J.~Vanderwerff and V.~Zizler, \emph{Biorthogonal Systems in Banach Spaces}, Springer 2008.

%\bibitem{KW} N.J. Kalton, D. Werner, \emph{Property $(M)$, $M$-ideals, and almost isometric structure of Banach spaces}, J.~Reine Angew. Math.~{\bf 461} (1995), 137--178. 

\bibitem{kuczma} M. Kuczma, \emph{An introduction to the theory of functional equations and inequalities. Cauchy's equation and Jensen's inequality}, 2$^{\mathrm{nd}}$ edition, edited by A.~Gil\'anyi, Birkh\"auser Verlag, Basel 2009.

\bibitem{lancien} G. Lancien, \emph{A~survey on the Szlenk index and some of its applications}, Rev. R.~Acad. Cien. Serie~A. Mat. {\bf 100} (2006), 209--235.

\bibitem{LT} J. Lindenstrauss, L. Tzafriri, \emph{Classical Banach spaces I. Sequence spaces}, Springer-Verlag 1977.

\bibitem{szlenk} W. Szlenk, \emph{The non-existence of a~separable reflexive Banach space universal for all separable reflexive Banach spaces}, Studia Math.~{\bf 30} (1968), 53--61.


\end{thebibliography}
\end{document}